\documentclass[12pt]{amsart}
\usepackage{}
\usepackage{amsmath}
\usepackage{amsfonts}
\usepackage{amssymb}
\usepackage[all,cmtip]{xy}           
\usepackage{xspace}

\usepackage{bbding}
\usepackage{txfonts}
\usepackage[shortlabels]{enumitem}
\usepackage{ifpdf}
\ifpdf
\usepackage[colorlinks,final,backref=page,hyperindex]{hyperref}
\else
\usepackage[colorlinks,final,backref=page,hyperindex,hypertex]{hyperref}
\fi
\usepackage{tikz}
\usepackage[active]{srcltx}

\makeatletter

\newtheorem{theorem}{Theorem}[section]
\newtheorem{prop}[theorem]{Proposition}
\newtheorem{lemma}[theorem]{Lemma}

\newtheorem{prop-def}{Proposition-Definition}[section]

\theoremstyle{definition}
\newtheorem{defn}[theorem]{Definition}

\newtheorem{exam}[theorem]{Example}

\newcommand{\nc}{\newcommand}

\nc{\delete}[1]{{}}
\nc{\mmargin}[1]{}

\nc{\mlabel}[1]{\label{#1}}  
\nc{\mcite}[1]{\cite{#1}}  
\nc{\mref}[1]{\ref{#1}}  
\nc{\meqref}[1]{\eqref{#1}} 
\nc{\mbibitem}[1]{\bibitem{#1}} 

\delete{
\nc{\mlabel}[1]{\label{#1}  
		{\hfill \hspace{1cm}{\bf{{\ }\hfill(#1)}}}}
\nc{\mcite}[1]{\cite{#1}{{\bf{{\ }(#1)}}}}  
\nc{\mref}[1]{\ref{#1}{{\bf{{\ }(#1)}}}}  
\nc{\meqref}[1]{\eqref{#1}{{\bf{{\ }(#1)}}}} 
\nc{\mbibitem}[1]{\bibitem[\bf #1]{#1}} 
}

\nc{\lush}{dense\xspace}

 \font\cyrs=wncyr7

\nc{\tforall}{\text{ for all }}

\newcommand{\bk}{{\mathbf{k}}}

\nc{\vep}{\varepsilon}
\nc{\bin}[2]{ (_{\stackrel{\scs{#1}}{\scs{#2}}})}  
\nc{\binc}[2]{(\!\! \begin{array}{c} \scs{#1}\\
		\scs{#2} \end{array}\!\!)}  
\nc{\bincc}[2]{  ( {\scs{#1} \atop
		\vspace{-1cm}\scs{#2}} )}  
\nc{\oline}[1]{\overline{#1}}
\nc{\mapm}[1]{\lfloor\!|{#1}|\!\rfloor}
\nc{\bs}{\bar{S}}
\nc{\la}{\longrightarrow}
\nc{\ot}{\otimes}
\nc{\rar}{\rightarrow}
\nc{\dar}{\downarrow}
\nc{\dap}[1]{\downarrow \rlap{$\scriptstyle{#1}$}}
\nc{\defeq}{\stackrel{\rm def}{=}}
\nc{\dis}[1]{\displaystyle{#1}}
\nc{\dotcup}{\ \displaystyle{\bigcup^\bullet}\ }
\nc{\hcm}{\ \hat{,}\ }
\nc{\hts}{\hat{\otimes}}
\nc{\hcirc}{\hat{\circ}}
\nc{\lleft}{[}
\nc{\lright}{]}
\nc{\curlyl}{\left \{ \begin{array}{c} {} \\ {} \end{array}
	\right .  \!\!\!\!\!\!\!}
\nc{\curlyr}{ \!\!\!\!\!\!\!
	\left . \begin{array}{c} {} \\ {} \end{array}
	\right \} }
\nc{\longmid}{\left | \begin{array}{c} {} \\ {} \end{array}
	\right . \!\!\!\!\!\!\!}
\nc{\ora}[1]{\stackrel{#1}{\rar}}
\nc{\ola}[1]{\stackrel{#1}{\la}}
\nc{\scs}[1]{\scriptstyle{#1}} \nc{\mrm}[1]{{\rm #1}}
\nc{\dirlim}{\displaystyle{\lim_{\longrightarrow}}\,}
\nc{\invlim}{\displaystyle{\lim_{\longleftarrow}}\,}
\nc{\dislim}[1]{\displaystyle{\lim_{#1}}} \nc{\colim}{\mrm{colim}}
\nc{\mvp}{\vspace{0.3cm}} \nc{\tk}{^{(k)}} \nc{\tp}{^\prime}
\nc{\ttp}{^{\prime\prime}} \nc{\svp}{\vspace{2cm}}
\nc{\vp}{\vspace{8cm}}
\nc{\modg}[1]{\!<\!\!{#1}\!\!>}
\nc{\intg}[1]{F_C(#1)}
\nc{\lmodg}{\!<\!\!}
\nc{\rmodg}{\!\!>\!}
\nc{\wit}[1]{\widetilde{#1}}
\nc{\cpi}{\widehat{\Pi}}
\nc{\ssha}{{\mbox{\cyrs X}}} 
\nc{\tsha}{{\mbox{\cyrt X}}}
\nc{\shpr}{\diamond}    
\nc{\labs}{\mid\!}
\nc{\rabs}{\!\mid}

\nc{\dfop}{\odot}
\nc{\dfoa}{\dfop^{(1)}}
\nc{\dfob}{\dfop^{(2)}}
\nc{\dfoc}{\dfop^{(3)}}
\nc{\dfod}{\dfop^{(4)}}
\nc{\tstar}{\tilde{\star}}
\nc{\tdfop}{\tilde{\dfop}}
\nc{\tdfoa}{\tilde{\dfop}^{(1)}}
\nc{\tdfob}{\tilde{\dfop}^{(2)}}
\nc{\tdfoc}{\tilde{\dfop}^{(3)}}
\nc{\tdfod}{\tilde{\dfop}^{(4)}}

\nc{\ann}{\mrm{ann}}
\nc{\Aut}{\mrm{Aut}}
\nc{\br}{\mrm{bre}}
\nc{\can}{\mrm{can}}
\nc{\coh}{\mrm{coh}}
\nc{\Cont}{\mrm{Cont}}
\nc{\rchar}{\mrm{char}}
\nc{\cok}{\mrm{coker}}
\nc{\de}{\mrm{dep}}
\nc{\dtf}{{R-{\rm tf}}}
\nc{\dtor}{{R-{\rm tor}}}

\nc{\Div}{{\mrm Div}}
\nc{\End}{\mrm{End}}
\nc{\Ext}{\mrm{Ext}}
\nc{\Fil}{\mrm{Fil}}
\nc{\Fr}{\mrm{Fr}}
\nc{\Frob}{\mrm{Frob}}
\nc{\Gal}{\mrm{Gal}}
\nc{\GL}{\mrm{GL}}
\nc{\Hom}{\mrm{Hom}}
\nc{\hsr}{\mrm{H}}
\nc{\hpol}{\mrm{HP}}
\nc{\id}{\mrm{id}}
\nc{\im}{\mrm{im}}
\nc{\Id}{\mrm{Id}}
\nc{\ID}{\mrm{ID}}
\nc{\Irr}{\mrm{Irr}}
\nc{\incl}{\mrm{incl}}
\nc{\length}{\mrm{length}}
\nc{\Lie}{\mrm{Lie}}
\nc{\mchar}{\rm char}
\nc{\mpart}{\mrm{part}}
\nc{\ql}{{\QQ_\ell}}
\nc{\qp}{{\QQ_p}}
\nc{\rank}{\mrm{rank}}
\nc{\rcot}{\mrm{cot}}
\nc{\rdef}{\mrm{def}}
\nc{\rdiv}{{\rm div}}
\nc{\rtf}{{\rm tf}}
\nc{\rtor}{{\rm tor}}
\nc{\res}{\mrm{res}}
\nc{\SL}{\mrm{SL}}
\nc{\Spec}{\mrm{Spec}}
\nc{\tor}{\mrm{tor}}
\nc{\Tr}{\mrm{Tr}}
\nc{\tr}{\mrm{tr}}

\nc{\bfk}{{\bf k}}
\nc{\bfone}{{\bf 1}}
\nc{\bfzero}{{\bf 0}}
\nc{\detail}{\marginpar{\bf More detail}
	\noindent{\bf Need more detail!}
	\svp}
\nc{\Diff}{\mathbf{Diff}}
\nc{\gap}{\marginpar{\bf Incomplete}\noindent{\bf Incomplete!!}
	\svp}
\nc{\FMod}{\mathbf{FMod}}
\nc{\Int}{\mathbf{Int}}
\nc{\Mon}{\mathbf{Mon}}
\nc{\remarks}{\noindent{\bf Remarks: }}
\nc{\Rep}{\mathbf{Rep}}
\nc{\Rings}{\mathbf{Rings}}
\nc{\Sets}{\mathbf{Sets}}
\nc{\sbu}{{\scriptstyle \bullet}}

\nc{\BA}{{\mathbb A}}   \nc{\BB}{{\mathbb B}}
\nc{\CC}{{\mathbb C}}   \nc{\DD}{{\mathbb D}}
\nc{\EE}{{\mathbb E}}
\nc{\FF}{{\mathbb F}}   \nc{\GG}{{\mathbb G}}
\nc{\HH}{{\mathbb H}}   \nc{\LL}{{\mathbb L}}
\nc{\NN}{{\mathbb N}}   \nc{\PP}{{\mathbb P}}
\nc{\QQ}{{\mathbb Q}}   \nc{\RR}{{\mathbb R}}
\nc{\BS}{{\mathbb S}}   \nc{\TT}{{\mathbb T}}
\nc{\VV}{{\mathbb V}}   \nc{\ZZ}{{\mathbb Z}}
\nc{\TP}{\widetilde{P}}

\nc{\cala}{{\mathcal A}}    \nc{\calb}{{\mathcal B}}
\nc{\calc}{{\mathcal C}}    \nc{\cald}{\mathcal{D}}
\nc{\cale}{{\mathcal E}}    \nc{\calf}{{\mathcal F}}
\nc{\calg}{{\mathcal G}}    \nc{\calh}{{\mathcal H}}
\nc{\cali}{{\mathcal I}}    \nc{\call}{{\mathcal L}}
\nc{\calm}{{\mathcal M}}    \nc{\caln}{{\mathcal N}}
\nc{\calo}{{\mathcal O}}    \nc{\calp}{{\mathcal P}}
\nc{\calq}{{\mathcal Q}}    \nc{\calr}{{\mathcal R}}
\nc{\cals}{{\mathcal S}}    \nc{\calt}{{\mathcal T}}
\nc{\calv}{{\mathcal V}}    \nc{\calw}{{\mathcal W}}
\nc{\calx}{{\mathcal X}}    \nc{\caly}{{\mathcal Y}}

\nc{\fraka}{{\mathfrak a}}
\nc{\frakb}{\mathfrak{b}}
\nc{\frakg}{{\frak g}}
\nc{\frakB}{{\frak B}}
\nc{\frakm}{{\frak m}}
\nc{\frakM}{{\frak M}}
\nc{\frakp}{{\frak p}}
\nc{\frakX}{{\frak X}}
\nc{\frakS}{{\frak S}}
\nc{\frakA}{{\frak A}}
\nc{\frakC}{{\frak C}}
\nc{\frakx}{{\frakx}}

\nc{\ynr}[1]{\textcolor{blue}{\underline{Yunnan:}#1 }}

\nc{\lir}[1]{\textcolor{red}{\underline{Li:}#1 }}

\begin{document}

\title[Unit action on braided operads and Hopf algebras]{Coherent unit actions on braided operads and Hopf algebras}

\author[Li Guo]{Li Guo}
\address{Department of Mathematics and Computer Science, Rutgers University, Newark, NJ 07102, USA}
\email{liguo@rutgers.edu}

\author[Yunnan Li]{Yunnan Li}
\address{School of Mathematics and Information Science, Guangzhou University, Waihuan Road West 230, Guangzhou 510006, China}
\email{ynli@gzhu.edu.cn}

\date{\today}

\begin{abstract}
The notion of a coherent unit action on algebraic operads was first introduced by Loday for binary quadratic nonsymmetric operads and generalized by Holtkamp, to ensure that the free objects of the operads carry a Hopf algebra structure. There was also a classification of such operads in the binary quadratic nonsymmetric case. We generalize the notion of coherent unit action to braided operads and show that free objects of braided operads with such an action carries a braided Hopf algebra structure. Under the conditions of binary, quadratic and nonsymmetric, we give a characterization and classification of the braided operads that allow a coherent unit action and thus carry a braided Hopf algebra structure on their free objects.
\end{abstract}

\subjclass[2010]{16T05, 16T25, 16W99, 05C05}

\keywords{operad, unit action, braided Hopf algebra, braided dendriform algebra, braided tridendriform algebra}

\maketitle

\tableofcontents

\allowdisplaybreaks

\section{Introduction}

This paper explores braided Hopf algebra structures on free objects from braided operads by the method of coherent unit actions.

As the importance of Hopf algebras continued to show in more and more areas of mathematics, ad hoc instances of Hopf algebras sometimes turned out to be just special cases of a general construction.
One such general construction is the free objects of various (nonassociative) algebras (operads). The first instance is the dendriform algebra introduced by Loday~\mcite{Lo1} who showed that, even though a dendriform algebra is not an associative algebra, free dendriform algebras has a natural Hopf algebra structure\mcite{LR1}. This Hopf algebra is realized on planar binary trees and was found to be isomorphic to the noncommutative analog (by Foissy and Holtkamp~\mcite{Fo1,Hol1}) of the Hopf algebra of rooted trees in the Connes-Kreimer theory of renormalization of quantum field theory~\mcite{CK}. After Hopf algebra structures were discovered for several related algebraic structures, such as tridendriform algebra and quadri-algebras~\mcite{AL,LR2}, Loday showed in~\mcite{Lo1} that a nonsymmetric (also called regular) operad with a so-called coherent unit action endows a natural Hopf algebra structure on its free objects. Similar results hold for the corresponding algebraic structures with certain commutativity~\mcite{Lo2,Lo3}. This uniform approach not only included as special cases many Hopf algebras obtained case by case during that period of time, it also recovered the classical Hopf algebra structures on the shuffle product algebra and quasi-shuffle algebra, as the algebras are shown to be the free objects of commutative dendriform and tridendriform algebras. This work was further generalized to nonsymmetric operads by Holtkamp~\mcite{Hol2}. Given the great interest in such Hopf algebras, a characterization and classification were achieved for binary quadratic nonsymmetric operads which have a coherent unit action, and hence allow for Hopf algebra structures on their free objects~\mcite{EG}.

Operads have been extended to various broader contexts. We are particularly interested in braided operads~\mcite{Fi,Fr} for their connection with Yang-Baxter equations and quantum theory\mcite{Ka}. In particular, in this context quantum shuffle algebras, quantum quasi-shuffle algebras and generally quantum multi-brace algebras have been obtained by Rosso, Jian-Rosso-Zhang and Jian-Rosso~\mcite{Ro,JRZ,Ji,JR} respectively. Recently, with motivation from braided construction of rooted trees from the work of Connes-Kreimer on renormalization of quantum field theory, braided structures for dendriform algebras and tridendriform algebras have attracted attention~\mcite{Fo2,GL,GL0}. There again the free objects can be equipped with braided Hopf algebra structures.

Thus it appears to be the time to provide a uniform approach in the context of braided operads, in the direction of coherent unit actions on algebraic operads developed by Loday and Holtkamp~\mcite{Lo2,Hol2} as noted above.
In particular, it is interesting to study the coherent unit actions such that the free objects of the braided operads are braided Hopf algebras. This is the purpose of this article. We also extend the classification of coherent unit actions on binary quadratic nonsymmetric operads in~\mcite{EG} to the braided context, thereby determining such braided operads that give rise to braided Hopf algebras from their free objects.

Here is the layout of the paper. In Section~\mref{sec:bra}, we provide background on braided operads with some details on the dendriform and tridendriform cases. In Section~\mref{sec:coh}, we extend the notion of coherent unit actions on nonsymmetric operads to braided operads and show that such braided operads have a braided Hopf algebra structure on their free objects. We consider the two cases of braided nonsymmetric and completely commutative operads, so that we can cover for instance both the braided dendriform operad and braided commutative dendriform (Zinbiel) operad. In Section~\mref{sec:clas}, we give a classification of braided binary quadratic nonsymmetric operads with coherent unit actions.

\smallskip

\noindent
{\bf Notations. } In this paper, we fix a ground field $\bk$ of characteristic 0. All the objects
under discussion, including vector spaces, algebras and tensor products, are taken over $\bk$ by default.
Also denote $\BS_n$ (resp. $\BB_n$) the $n$-th symmetric group (resp. braid group) for any $n\geq0$ with $\BS_0=\emptyset$ (resp. $\BB_0=\emptyset$).
For the convenience of discussion, we also fix a braided tensor category $\calc$ over $\bk$ to involve all the forthcoming braided objects.

\section{Braided algebraic operads}
\mlabel{sec:bra}
We start with a basic notion.
\begin{defn}
A {\bf braided vector space} is a vector space $V$ together with a linear operator $\sigma$ on $V^{\otimes2}$, called the {\bf Yang-Baxter operator} characterized by the equation
\begin{equation}
(\sigma\otimes
\id_{V})(\id_{V}\otimes \sigma)(\sigma\otimes
\id_{V})=(\id_{V}\otimes \sigma)(\sigma\otimes
\id_{V})(\id_{V}\otimes \sigma).
\mlabel{eq:qybe}
\end{equation}
For any braided vector space $(V,\sigma)$ and $n\geq1$, the tensor space $V^{\otimes n}$ becomes a representation of $\BB_n$ with its usual generators $b_i$ acting as $\sigma_i:=\id^{\otimes(i-1)}\otimes\sigma\otimes\id^{\otimes(n-1-i)}$ for $1\leq i\leq n-1$. For braided vector spaces $(V,\sigma),(V',\sigma')$, a linear map $f:V\rar V'$ is called a {\bf homomorphism of braided vector spaces} if $(f\otimes f)\sigma=\sigma'(f\otimes f)$.
Let {\bf VS} (resp. {\bf BV}) denote the category of (braided) vector spaces.
\end{defn}

Let $\pi_n:\BB_n\rar\BS_n$ be the natural projection mapping $b_i$ to transpositions $s_i,\,i=1,\dots,n-1$.
It is well known that there exists an injective map $\iota_n:\BS_n\rar\BB_n$, satisfying $\pi_n\iota_n=\id_{\BS_n}$ and called the {\bf Matsumoto-Tits section}. If $w=s_{i_1}\cdots s_{i_r}$ is
any reduced expression of $w\in\BS_n$, then $\iota_n(w)=b_{i_1}\cdots b_{i_r}$, which is uniquely determined by $w$. Given any braided vector space $(V,\sigma)$, we denote the action of $\iota_n(w)$ on $V^{\otimes n}$ by $T^\sigma_w:=\sigma_{i_1}\cdots
\sigma_{i_r}$. In particular, for the usual flip map $\tau$ of $V$, we have
\[T^\tau_w(v_1\otimes\cdots\otimes v_n)=v_{w^{-1}(1)}\otimes\cdots\otimes v_{w^{-1}(n)}.\]

\begin{defn}
	Let $A$ be an algebra with product $\mu$, and $\sigma$ be a braiding on $A$. We call the triple $(A,\mu,\sigma)$ a {\bf braided algebra} if it satisfies the conditions
	\begin{equation}\mlabel{ba1}
	 (\mu\otimes\Id_A)\sigma_2\sigma_1=\sigma(\Id_A\otimes\mu),\quad
	(\Id_A\otimes\mu)\sigma_1\sigma_2=\sigma(\mu\otimes
	\Id_A).
	\end{equation}
	Moreover, if $A$ is unital with unit $1_A$ and satisfies
	\begin{equation}\mlabel{ba2}
	\sigma(a\otimes 1_A)=1_A\otimes a,\quad
	\sigma(1_A\otimes a)=a\otimes1_A \quad \text{for all } a\in A, \end{equation}
	then $A$ is called a {\bf unital braided algebra}.

Dually, we call $(C,\Delta,\varepsilon,\sigma)$ a {\bf braided coalgebra} if $(C,\Delta,\varepsilon)$ is a coalgebra with braiding $\sigma$ and satisfies
	\begin{align}
	&\mlabel{bc1}
\sigma_1\sigma_2(\Delta\otimes\Id_C)=(\Id_C\otimes\Delta)\sigma,\,
	\sigma_2\sigma_1(\Id_C\otimes\Delta)=(\Delta\otimes
	\Id_C)\sigma,\\
    &\mlabel{bc2}
	 (\varepsilon\otimes\Id_C)\sigma=\Id_C=(\Id_C\otimes\varepsilon)\sigma.
	\end{align}
	
Also, a quintuple $(H,\mu,\Delta,\varepsilon,\sigma)$ is called a {\bf braided bialgebra}, if $(H,\mu,\sigma)$ (resp. $(H,\Delta,\varepsilon,\sigma)$) is a braided algebra (resp. coalgebra) such that
	\begin{equation}
	\Delta\,\mu=(\mu\otimes\mu)\sigma_2(\Delta\otimes \Delta).
	\mlabel{eq:comp}
	\end{equation}

When $H$ has unit $1_H$ and antipode $S:H\to H$ such that $\mu(S\otimes\Id_H)\Delta=\varepsilon 1_H=\mu(\Id_H\otimes S)\Delta$,
	then the septuple $(H,\mu,\Delta,\vep,S,\sigma)$ is a {\bf braided Hopf algebra}. Similar to the case of braided vector spaces, the homomorphisms for these braided objects are the usual homomorphisms commutating with braidings.
\end{defn}

The notion of braided operads was originally introduced by Fiedorowicz in \mcite{Fi}. Especially when the Yang-Baxter operators are chosen to be the flip maps, the framework recovers the usual definition of operads. For the study of algebraic operads and braided operads, we refer the reader to \cite{Fr,LV}.

\begin{defn}
A {\bf braided (algebraic) operad} over $\bk$ is an analytic functor $\calp:\rm{\bf BV}\rar\rm{\bf BV}$ such that
$\calp(0)=0$ and is equipped with an associative natural transformation of functors $\gamma:\calp\circ\calp\rar\calp$ and a unit $\eta:\id_{\rm{\bf BV}}\rar\calp$ such that the following diagrams commute:
\[\xymatrix@=2.5em{(\calp\circ\calp)\circ\calp\ar@{->}[r]^-{=}\ar@{->}[d]^-{\gamma\circ\id_\calp}&\calp\circ(\calp\circ\calp)\ar@{->}[r]^-{\id_\calp\circ\gamma}
&\calp\circ\calp\ar@{->}[d]^-\gamma\\
\calp\circ\calp\ar@{->}[rr]^-\gamma&&\calp}\,,\quad\,
\xymatrix@=2.5em{\id_{\rm{\bf BV}}\circ\calp\ar@{->}[r]^-{\eta\circ\id_\calp}\ar@{->}[dr]_{=}&\calp\circ\calp\ar@{->}[d]^-\gamma&\calp\circ\id_{\rm{\bf BV}}\ar@{->}[l]_{\id_\calp\circ\eta}\ar@{->}[dl]^-{=}\\
&\calp&}\,.\]
Namely,
$$\gamma_V\gamma_{\calp(V)}=\gamma_V\calp(\gamma_V), \quad  \gamma_V\eta_{\calp(V)}=\id_{\calp(V)}=\gamma_V\calp(\eta_V) \quad \text{for any } (V,\sigma)\in {\bf BV}.$$
For two braided operads $(\calp,\gamma_\calp,\eta_\calp),\,(\calq,\gamma_\calq,\eta_\calq)$, a natural transformation $\alpha:\calp\rar\calq$ is called a {\bf morphism of braided operads}, if $\gamma_\calq\circ(\alpha,\alpha)=\alpha\circ\gamma_\calp$ and $\eta_\calq=\alpha\circ\eta_\calp$.
\end{defn}

\begin{defn}
A {\bf $\BB$-module} over $\bk$ is a family
$M=\{M(n)\}_{n\geq0}$
of right $\BB_n$-modules $M(n)$. A morphism of $\BB$-modules $f:M\rar N$ is a family of homomorphisms of $\BB_n$-modules
$f_n:M(n)\rar N(n)$.
\end{defn}
Given a $\BB$-module $M$ and a braided vector space $(V,\sigma)$, we define a functor $\wit{M}:\text{\bf BV}\rar\text{\bf BV}$ by
\[\wit{M}(V):=\bigoplus_{n\geq0}M(n)\otimes_{\BB_n}V^{\otimes n},\]
whose Yang-Baxter operator, denoted by $\sigma_M$, is determined by the following equalities,
\begin{equation}\mlabel{eq:myb}
\sigma_M((\mu,u_1,\dots,u_i)\otimes(\nu,v_1,\dots,v_j)) :=\sum(\nu,v'_1,\dots,v'_j)
\otimes(\mu,u'_1,\dots,u'_i),
\end{equation}
for any $(\mu,u_1,\dots,u_i)\in M(i)\otimes_{\BB_i}V^{\otimes i}$ and $(\nu,v_1,\dots,v_j)\in M(j)\otimes_{\BB_j}V^{\otimes j}$, where we denote
\[T^\sigma_{\beta_{ij}}(u_1\otimes\cdots\otimes u_i\underline{\otimes} v_1\otimes\cdots\otimes v_j):=\sum v'_1\otimes\cdots\otimes v'_j\underline{\otimes} u'_1\otimes\cdots\otimes u'_i,\]
with $\beta_{ij}\in\BS_{i+j}$ such that $\beta_{ij}(k)=j+k$ if $1\leq k\leq i$, and $k-i$ if $i+1\leq k\leq i+j$.

Define the {\bf tensor product} of $\BB$-modules $M$ and $N$ to be the $\BB$-module $M\otimes N$ with
\[(M\otimes N)(n):=\bigoplus_{i+j=n}\text{Ind\,}_{\BB_i\times\BB_j}^{\BB_n} M(i)\otimes N(j)=
\bigoplus_{i+j=n}(M(i)\otimes N(j))\otimes_{\BB_i\times\BB_j}\bk[\BB_n],\,n\geq0.\]

Define the {\bf composite} of $\BB$-modules $M$ and $N$ to be the $\BB$-module
\[M\circ N:=\bigoplus_{k\geq0}M(k)\otimes_{\BB_k} N^{\otimes k}.\]
More precisely,
$$(M\circ N)(n):=\bigoplus_{k\geq0}M(k)\otimes_{\BB_k}
\left(\bigoplus_{i_1+\cdots+i_k=n}\text{Ind\,}_{\BB_{i_1}\times\cdots\times\BB_{i_k}}^{\BB_n}
N(i_1)\otimes\cdots\otimes N(i_k)\right),\,n\geq0,$$
where the left module action of $\BB_k$ on $N^{\otimes k}$ is defined as follows. Any braid $b\in\BB_k$ sends
\[(\mu_1,\dots,\mu_k,c)\in
(N(i_1)\otimes\cdots\otimes N(i_k))\otimes_{\BB_{i_1}\times\cdots\times\BB_{i_k}}\bk[\BB_n]\]
to
\[(\mu_{b^{-1}(1)},\dots,\mu_{b^{-1}(k)},b(i_1,\dots,i_k)c)\in
(N(i_{b^{-1}(1)})\otimes\cdots\otimes N(i_{b^{-1}(k)}))
\otimes_{\BB_{i_{b^{-1}(1)}}\times\cdots\times\BB_{i_{b^{-1}(k)}}}\bk[\BB_n],\]
where $b\in\BB_k$ acts on $\{1,\dots,k\}$ by permutations via the natural projection $\pi_k$, and $b(i_1,\dots,i_k)$ is the braid obtained from $b$ by replacing its $j$-th strand with $i_j$ parallel strands for any $j=1,\dots,k$.

\begin{prop}
For any $\BB$-modules $M,\,N$ and braided vector space $V$, we have
\[(\wit{M\otimes N})(V)=\wit{M}(V)\otimes\wit{N}(V),\,
(\wit{M\circ N})(V)=\wit{M}(\wit{N}(V)).\]
\end{prop}
\begin{proof}
By the definition of $\wit{M}$ and $M\otimes N$, we find
\begin{align*}
(\wit{M\otimes N})(V)&=\bigoplus_{n\geq0}(M\otimes N)(n)\otimes_{\BB_n}V^{\otimes n}\\
&=\bigoplus_{n\geq0}\left(\bigoplus_{i+j=n}(M(i)\otimes N(j))\otimes_{\BB_i\times\BB_j}\bk[\BB_n]\right)\otimes_{\BB_n}V^{\otimes n}\\
&=\bigoplus_{i,j\geq0}(M(i)\otimes N(j))\otimes_{\BB_i\times\BB_j}V^{\otimes(i+j)}\\
&=\left(\bigoplus_{i\geq0}M(i)\otimes_{\BB_i}V^{\otimes i}\right)\otimes\left(\bigoplus_{j\geq0}N(j)\otimes_{\BB_j}V^{\otimes j}\right)\\
&=\wit{M}(V)\otimes\wit{N}(V).
\end{align*}
On the other hand, by the definition of $\wit{M}$ and $M\circ N$, we have
\begin{align*}
(\wit{M\circ N})(V)&=\bigoplus_{n\geq0}(M\circ N)(n)\otimes_{\BB_n}V^{\otimes n}\\
&=\bigoplus_{n\geq0}\left(\bigoplus_{k\geq0}M(k)\otimes_{\BB_k}\left(\bigoplus_{i_1+\cdots+i_k=n}(N(i_1)\otimes\cdots\otimes N(i_k))\otimes_{\BB_{i_1}\times\cdots\times\BB_{i_k}}\bk[\BB_n]\right)\right)\otimes_{\BB_n}V^{\otimes n}\\
&=\bigoplus_{k\geq0}M(k)\otimes_{\BB_k}\left(\bigoplus_{i_1,\dots,i_k\geq0}(N(i_1)\otimes\cdots\otimes N(i_k))\otimes_{\BB_{i_1}\times\cdots\times\BB_{i_k}}V^{\otimes(i_1+\cdots+i_k)}\right)\\
&=\bigoplus_{k\geq0}M(k)\otimes_{\BB_k}\left(\left(\bigoplus_{i_1\geq0}N(i_1)\otimes_{\BB_{i_1}}V^{\otimes i_1}\right)\otimes\cdots\otimes
\left(\bigoplus_{i_k\geq0}N(i_k)\otimes_{\BB_{i_k}}V^{\otimes i_k}\right)\right)\\
&=\bigoplus_{k\geq0}M(k)\otimes_{\BB_k}\wit{N}(V)^{\otimes k}\\
&=\wit{M}(\wit{N}(V)),
\end{align*}
where the third equality is due to the fact that the left action of $\BB_k$ on $\text{Ind\,}_{\BB_{i_1}\times\cdots\times\BB_{i_k}}^{\BB_n}
N(i_1)\otimes\cdots\otimes N(i_k)$ commutes with the right one of $\BB_n$, and the fifth equality is based on the representation $\wit{N}(V)^{\otimes k}$ of $\BB_k$ via the Yang-Baxter operator $\sigma_N$ of $\wit{N}(V)$.
\end{proof}

\medskip
According to \cite[Definition~3.2]{Fi} or the construction in \cite[\S 5.1]{Fr}, a braided operad $\calp$ consists of a
$\BB$-module $\{\calp(n)\}_{n\geq0}$,
with the composition maps $$\gamma(i_1,\dots,i_k):\calp(k)\otimes\calp(i_1)\otimes\cdots\otimes\calp(i_k)
\rar\calp(i_1+\cdots+i_k)$$
and the unit map $\eta$ satisfying
the associativity and the unity conditions, and also the following equivalence conditions:
\begin{align*}
&\gamma(\mu b;\mu_1,\dots,\mu_k)=\gamma(\mu;\mu_{b^{-1}(1)},\dots,\mu_{b^{-1}(k)})b(i_1,\dots,i_k),\\
&\gamma(\mu;\mu_1b_1,\dots,\mu_kb_k)=\gamma(\mu;\mu_1,\dots,\mu_k)(b_1\times\cdots\times b_k),
\end{align*}
for any $\mu\in\calp(k),\mu_1\in\calp(i_1),\dots,\mu_k\in\calp(i_k)$, where $b\in\BB_k$ acts on $\{1,\dots,k\}$ by permutations via the natural projection $\pi_k$, $b(i_1,\dots,i_k)$ is the braid obtained from $b$ by replacing its $j$-th strand with $i_j$ parallel strands for any $j=1,\dots,k$, and $b_1\times\cdots\times b_k$ denotes the direct product of braids $b_1,\dots,b_k$.
It is easy to see that $\gamma(i_1,\dots,i_k)$ factors through
\[\calp(k)\otimes_{\BB_k}\text{Ind\,}_{\BB_{i_1}\times\cdots\times\BB_{i_k}}^{\BB_{i_1+\cdots+i_k}}
\calp(i_1)\otimes\cdots\otimes\calp(i_k)=\calp(k)\otimes_{\BB_k}(\calp(i_1)\otimes\cdots\otimes\calp(i_k))\otimes_{\BB_{i_1}\times\cdots\times\BB_{i_k}}\bk[\BB_{i_1+\cdots+i_k}]\]
by the homomorphism of $\BB_{i_1+\cdots+i_k}$-modules sending any
$\mu\otimes_{\BB_k}(\mu_1\otimes\cdots\otimes\mu_k)\otimes_{\BB_{i_1}\times\cdots\times\BB_{i_k}}b$ to $\gamma(\mu;\mu_1,\dots,\mu_k)b$,
as we have $b(i_1,\dots,i_k)(b_1\times\cdots\times b_k)=(b_{b^{-1}(1)}\times\cdots\times b_{b^{-1}(k)})b(i_1,\dots,i_k)$.

Given any $(V,\sigma)\in${\bf BV}, we have
\[\calp(V):=\bigoplus_{n\geq0}\calp(n)\otimes_{\BB_n}V^{\otimes n},\]
with its Yang-Baxter operator $\sigma_\calp$ defined as in Eq.~\meqref{eq:myb}. If $\calp $ is {\bf nonsymmetric} (also called {\bf regular}), then $\calp(n)=\calp_n\otimes\bk[\BB_n],\,n\geq0$, where $\bigoplus_{n\geq0}\calp_n$ is a graded vector space. Hence, $\calp(V)=\bigoplus_{n\geq0}\calp_n\otimes V^{\otimes n}$.

\begin{defn}
Given any (braided) operad $\calp$, a (braided) vector space $A$ is called an {\bf algebra over $\calp$} or {\bf $\calp$-algebra}, if it is equipped with a homomorphism $\theta_A:\calp(A)\rar A$ of (braided) vector spaces such that the following diagrams commute:
\[\xymatrix@=2.5em{(\calp\circ\calp)(A)\ar@{->}[r]^-{=}\ar@{->}[d]^-{\gamma_A}&\calp(\calp(A))\ar@{->}[r]^-{\theta_{\calp(A)}}
&\calp(A)\ar@{->}[d]^-{\theta_A}\\
\calp(A)\ar@{->}[rr]^-{\theta_A}&&A}\,,\quad\,
\xymatrix@=2.5em{A\ar@{->}[r]^-{\eta_A}\ar@{->}[rd]_{=}&\calp(A)\ar@{->}[d]^-{\theta_A}\\
&A}\,.\]
For two $\calp$-algebras $A$ and $B$, a homomorphism $\varphi:A\rar B$ of (braided) vector spaces is called a {\bf morphism of $\calp$-algebras} if $\varphi\theta_A=\theta_B\calp(\varphi)$.

In particular, the {\bf free $\calp$-algebra} generated by a (braided) vector space $V$ is given by $(\calp(V),\gamma_V)$ with $\eta_V:V\rar\calp(V)$, such that
for any $\calp$-algebra $A$ and homomorphism $f:V\rar A$ of (braided) vector spaces, there exists a unique morphism $\tilde{f}:\calp(V)\rar A$ of $\calp$-algebras satisfying $f=\tilde{f}\eta_V$.
\end{defn}

\begin{exam}\mlabel{ex:den}
A dendriform algebra $D$ is a vector space with two binary operators $\prec,\succ$ such that
\begin{align*}
&(a\prec b)\prec c=a\prec(b\prec c+b\succ c),\\
&(a\succ b)\prec c=a\succ(b\prec c),\\
&a\succ(b\succ c)=(a\prec b+a\succ b)\succ c,
\end{align*}
for any $a,b,c\in D$. Intrinsically, any dendriform algebra is
an algebra over the nonsymmetric operad $(\calp_\cald,\gamma,I)$ generated by
\[\calp_{\cald,1}=\bk I,\,\calp_{\cald,2}=\bk\prec\oplus\,\bk\succ,\]
with relations
\begin{align*}
&\gamma(\prec;\prec,I)=\gamma(\prec;I,\prec+\succ),\\
&\gamma(\prec;\succ,I)=\gamma(\succ;I,\prec),\\
&\gamma(\succ;I,\succ)=\gamma(\succ;\prec+\succ,I),\\
&\gamma(\prec;I,I)=\prec,\,\gamma(\succ;I,I)=\succ,\\
&\gamma(I;\prec)=\prec,\,\gamma(I;\succ)=\succ.
\end{align*}

As a braided analogue, we interpret the braided dendriform algebras in \mcite{GL} as
the braided $\calp_\cald$-algebras. A {\bf braided dendriform algebra} $(D,\prec,\succ,\sigma)$
is a braided vector space $(D,\sigma)$ endowed with the dendriform algebra structure $(\prec,\succ)$ such that
\begin{align}\mlabel{bda1}
\sigma(\Id_D\otimes\prec)=(\prec\otimes\Id_D)\sigma_2\sigma_1,\quad
\sigma(\prec\otimes\Id_D)=(\Id_D\otimes\prec)\sigma_1\sigma_2,\\
\mlabel{bda2}
\sigma(\Id_D\otimes\succ)=(\succ\otimes\Id_D)\sigma_2\sigma_1,\quad
\sigma(\succ\otimes\Id_D)=(\Id_D\otimes\succ)\sigma_1\sigma_2.
\end{align}
Since the map $\theta_D:\calp_\cald(D)\rar D$ is a homomorphism of braided vector spaces, we clearly have
\[(\theta_D\otimes\theta_D)\sigma_{\calp_\cald}=\sigma(\theta_D\otimes\theta_D).\]
Then for any $(I,a)\in\calp_{\cald,1}\otimes D$ and $(\prec,b,c)\in\calp_{\cald,2}\otimes D^{\otimes 2}$, we have
\[(\theta_D\otimes\theta_D)\sigma_{\calp_\cald}((I,a)\otimes(\prec,b,c))
=(\theta_D\otimes\theta_D) \left(\sum(\prec,b',c')\otimes(I,a')\right)
=\sum(b'\prec c')\otimes a',\]
where we set $T_{\beta_{12}}^\sigma(a\otimes b\otimes c)=\sigma_2\sigma_1(a\otimes b\otimes c)=\sum b'\otimes c'\otimes a'$.
On the other hand,
\[\sigma(\theta_D\otimes\theta_D)((I,a)\otimes(\prec,b,c))=\sigma(a\otimes(b\prec c)).\]
Hence,
\[\sigma(a\otimes(b\prec c))=\sum(b'\prec c')\otimes a',\]
that is, $\sigma(\Id_D\otimes\prec)=(\prec\otimes\Id_D)\sigma_2\sigma_1$.
The other conditions in Eq.~\meqref{bda1}, \meqref{bda2} can be verified similarly.

In particular, the free $\calp_\cald$-algebra over a braided vector space $(V,\sigma)$ is the free braided dendriform algebra, realized as the braided analogue of the Loday-Ronco algebra of planar binary rooted trees; see \cite[Theorem 2.8]{GL0}.
\end{exam}

\begin{exam}\mlabel{ex:triden}
A tridendriform algebra $T$ is a vector space with three binary operators $\prec,\succ,*$ such that
\begin{align*}
&(a\prec b)\prec c= a\prec(b\prec c+b\succ c+\,b* c),\\
&(a\succ b)\prec c= a\succ (b \prec c),\\
&a\succ (b\succ c)=(a\prec b+a\succ b+\,a* b)\succ c,\\
&(a* b)\prec c= a* (b\prec c),\\
&(a\prec b)* c= a* (b\succ c),\\
&(a\succ b)* c= a\succ(b * c),\\
&(a* b)* c= a*(b* c),
\end{align*}
for $a,b,c\in T$. Intrinsically, any tridendriform algebra is
an algebra over the nonsymmetric operad $(\calp_\calt,\gamma,I)$ generated by
\[\calp_{\calt,1}=\bk I,\,\calp_{\calt,2}=\bk\prec\oplus\,\bk\succ\oplus\,\bk\,*,\]
with relations
\begin{align*}
&\gamma(\prec;\prec,I)=\gamma(\prec;I,\prec+\succ+\,*),\\
&\gamma(\prec;\succ,I)=\gamma(\succ;I,\prec),\\
&\gamma(\succ;I,\succ)=\gamma(\succ;\prec+\succ+\,*,I),\\
&\gamma(\prec;*,I)=\gamma(*;I,\prec),\\
&\gamma(*;\prec,I)=\gamma(*;I,\succ),\\
&\gamma(*;\succ,I)=\gamma(\succ;I,*),\\
&\gamma(*;*,I)=\gamma(*;I,*),\\
&\gamma(\prec;I,I)=\prec,\,\gamma(\succ;I,I)=\succ,\,\gamma(*;I,I)=*,\\
&\gamma(I;\prec)=\prec,\,\gamma(I;\succ)=\succ,\,\gamma(I;*)=*.
\end{align*}

A {\bf braided tridendriform algebra}~\mcite{GL0}, denoted $(T,\prec,\succ,*,\sigma)$
is a braided vector space $(T,\sigma)$ endowed with a tridendriform algebra structure $(\prec,\succ,*)$ such that
\begin{align}\mlabel{bta1}
\sigma(\Id_T\otimes\prec)=(\prec\otimes\Id_T)\sigma_2\sigma_1,\quad
\sigma(\prec\otimes\Id_T)=(\Id_T\otimes\prec)\sigma_1\sigma_2,\\
\mlabel{bta2}
\sigma(\Id_T\otimes\succ)=(\succ\otimes\Id_T)\sigma_2\sigma_1,\quad
\sigma(\succ\otimes\Id_T)=(\Id_T\otimes\succ)\sigma_1\sigma_2,\\
\mlabel{bta3}
\sigma(\Id_T\otimes*)=(*\otimes\Id_T)\sigma_2\sigma_1,\quad
\sigma(*\otimes\Id_T)=(\Id_T\otimes*)\sigma_1\sigma_2.
\end{align}
Then the braided tridendriform algebras are exactly
the braided $\calp_\calt$-algebras, as the conditions in \meqref{bta1}--\meqref{bta3} can be recovered by the fact that the map $\theta_T:\calp_\calt(T)\rar T$ is a homomorphism of braided vector spaces. In particular, the free $\calp_\calt$-algebra over a braided vector space $(V,\sigma)$ is the free braided tridendriform algebra, constructed as the braided analogue of the Loday-Ronco algebra of planar rooted trees; see \cite[Theorem 4.5]{GL0}.
\end{exam}

\section{Coherent unit actions and braided $\calp$-Hopf algebras}
\mlabel{sec:coh}
Let $\calp$ be any (braided) operad such that $\calp(0)=0$ and $\calp(1)=\bk=\bk I$. Let $\cali$ be a 0-ary element adjoined to $\calp$ by
\[\calp'(i):=\begin{cases}
\calp(i),&i\geq1,\\
\bk\cali,&i=0.
\end{cases}\]
A {\bf unit action} on the operad $(\calp,\gamma,I)$ is a partial extension of the composition map $\gamma$ on $\calp'$ with
\[\gamma(i_1,\dots,i_k):\calp'(k)\otimes_{\BB_k}\text{Ind\,}_{\BB_{i_1}\times\cdots\times\BB_{i_k}}^{\BB_{i_1+\cdots+i_k}}
\calp'(i_1)\otimes\cdots\otimes\calp'(i_k)\rar\calp'(i_1+\cdots+i_k)\]
defined for all $i_j\geq0$ with $j=1,\dots,k$ and $i_1+\cdots+i_k>0$ if $k\geq2$. Note that $\gamma(i_1,\cdots,i_k)$ is not defined when $i_1=\cdots=i_k=0$ and $k\geq 2$.

In particular, for
$\gamma(1,0):\calp_2\otimes\bk I\otimes \bk\cali\rar\bk I$ and $\gamma(0,1):\calp_2\otimes \bk\cali\otimes\bk I\rar\bk I$ in the nonsymmetric case,
there exist linear maps $\alpha,\beta:\calp_2\rar\bk$ such that
\begin{equation}\mlabel{eq:ua}
\gamma(\mu;I,\cali)=\alpha(\mu)I,\,\gamma(\mu;\cali,I) =\beta(\mu)I \quad \text{ for any }\mu\in\calp_2.
\end{equation}

Given a $\calp$-algebra $A$, we define $A^+:=\bk\oplus A$, called a {\bf unitary} $\calp$-algebra, with structure map $\theta_{A^+}:\calp'(A^+)\rar A^+$ extending $\theta_A$ by sending $\cali\in\calp'(0)$ to $1\in A^+$.
\begin{defn}
For any (braided) nonsymmetric operad $(\calp,\gamma,I)$ generated by operation sets $M_k\subseteq\calp_k$, $k\geq2$, with a unit action, suppose that there are operations $\star_n\in \calp_n$ for all $n\geq0$ satisfying
\begin{align}
&\mlabel{eq:starn}\gamma(\star_n;I,\dots,\stackrel{i\text{ th}}{\vspace{3mm}\cali},\dots,I)=\star_{n-1},\,1\leq i\leq n,\\
&\mlabel{eq:star2}\gamma(\star_2;\cali,I)=\gamma(\star_2;I,\cali)=I,
\end{align}
particularly $\star_0=\cali,\,\star_1=I$ and $\alpha(\star_2)=\beta(\star_2)=1$. Then one can further extend $\gamma$ by requiring
\[\gamma(\star_n;\overbrace{\cali,\dots,\cali}^{n\text{ times}})=\cali.\]
In this case, such a unit action is called {\bf coherent}, if for any two $\calp$-algebras $A$ and $B$,
\[A\boxtimes B:=(A\otimes\bk)\oplus(\bk\otimes B)\oplus(A\otimes B)\]
is again a $\calp$-algebra defined as follows. Let $c$ be the braiding in the braided tensor category $\calc$; see \cite[Ch. XIII]{Ka}.
We let
$$\sigma:=\sigma_{A,B}:=c_{A^+\oplus B^+,\,A^+\oplus B^+}$$ denote the corresponding braiding on $A^+\oplus B^+$ when $A,B\in\calc$.
For any $p\in M_n$, $a_i\in A^+$ and $b_i\in B^+,\,1\leq i\leq n$ with $n\geq2$, let
\begin{equation}\mlabel{eq:tpp-alg}
p(a_1\otimes b_1,\dots,a_n\otimes b_n):=\begin{cases}
\sum\star_n(a'_1,\dots,a'_n)\otimes p(b'_1,\dots,b'_n),&\text{if at least one }b_i\in B,\\
p(a_1,\dots,a_n)\otimes 1,&\text{if all }b_i=1,\,p(a_1,\dots,a_n)\text{ is defined},\\
\text{undefined},&\text{otherwise},
\end{cases}
\end{equation}
where we denote $T^\sigma_{w_n}(a_1\otimes b_1\underline{\otimes}\cdots\underline{\otimes} a_n\otimes b_n)=\sum a'_1\otimes \cdots\otimes a'_n\otimes b'_1\otimes \cdots\otimes b'_n$ with $w_n\in\BS_{2n}$ such that $w_n(2i-1)=i,\,w_n(2i)=n+i$ for $i=1,\dots,n$. Then Eq.~\meqref{eq:tpp-alg} actually defines the structure map $\theta_{A\boxtimes B}$ of $A\boxtimes B$ as a $\calp$-algebra.
\end{defn}

\begin{lemma}\mlabel{le:ass}
For a unit action on a braided operad $(\calp,\gamma,I)$, let $\star_0:=\cali,\,\star_1:=I$, and $\star_2\in\calp_2$ satisfying Eq.~\meqref{eq:star2}. If $\star_2$ is associative, i.e. $\gamma(\star_2;\star_2,I)=\gamma(\star_2;I,\star_2)$, then one can recursively define a sequence of operations $\star_n,\,n\geq0$, by
\[\star_n:=\gamma(\star_2;\star_{n-1},I)\in\calp_n,\]
satisfying Eq.~\meqref{eq:starn}. Moreover,
\begin{equation}\mlabel{eq:ass}
\gamma(\star_2;\star_i,\star_{n-i})=\star_n,
\end{equation}
for all $i=0,1,\dots,n$ with $n\geq0$. Then we call an operad $(\calp,\gamma,I)$ {\bf associative}, if it has such an associative operation $\star:=\star_2\in\calp_2$.
\end{lemma}
\begin{proof}
First we inductively prove that $\star_n,\,n\geq0$, satisfy Eq.~\meqref{eq:starn}. It is clear when $n=1,2$. If $n>2$, then by the definition of $\star_n$ and the induction hypothesis we have
\begin{align*}
\gamma(\star_n;I,\dots,\stackrel{i\text{ th}}{\vspace{3mm}\cali},\dots,I)
&=\gamma(\gamma(\star_2;\star_{n-1},I);I,\dots,\stackrel{i\text{ th}}{\vspace{3mm}\cali},\dots,I)\\
&=\begin{cases}
\gamma(\star_2;\gamma(\star_{n-1};I,\dots,\cali,\dots,I),I),&1\leq i<n,\\
\gamma(\star_2;\gamma(\star_{n-1};I,\dots,I),\cali),&i=n,
\end{cases}\\
&=\begin{cases}
\gamma(\star_2;\star_{n-2},I),&1\leq i<n,\\
\gamma(\star_2;\star_{n-1},\cali),&i=n,
\end{cases}\\
&=\begin{cases}
\gamma(\star_2;\star_{n-2},I),&1\leq i<n,\\
\gamma(\gamma(\star_2;I,\cali);\star_{n-1}),&i=n,
\end{cases}\\
&=\star_{n-1}.
\end{align*}
Further, Eq.~\meqref{eq:ass} can also be proved by induction on $n$. Indeed, it is obviously true when $n=0,1,2$. If $n>2$, then
\begin{align*}
\gamma(\star_2;\star_i,\star_{n-i}) &=\gamma(\star_2;\gamma(\star_2;\star_{i-1},I),\star_{n-i})\\
&=\gamma(\gamma(\star_2;\star_2,I);\star_{i-1},I,\star_{n-i})\\
&=\gamma(\gamma(\star_2;I,\star_2);\star_{i-1},I,\star_{n-i})\\
&=\gamma(\star_2;\star_{i-1},\gamma(\star_2;I,\star_{n-i}))\\
&=\gamma(\star_2;\star_{i-1},\star_{n+1-i}),
\end{align*}
for any $i=2,\dots,n-1$. It then follows that
\[\star_n=\gamma(\star_2;\star_n,\cali)=\gamma(\star_2;\star_{n-1},I)
=\cdots=\gamma(\star_2;\cali,\star_n).\]
Hence, Eq.~\meqref{eq:ass} holds for any $n\geq2$.
\end{proof}

\begin{exam}\mlabel{ex:denua}
Let $\calp_\cald$ be the braided dendriform operad given in Example \mref{ex:den}. There is a unit action on $\calp_\cald$ defined as follows.
First extend $\calp_\cald$ to be $\calp'_\cald$ such that
\[\calp'_\cald(0)=\bk\cali,\,\calp'_{\cald,1}=\bk I,\,\calp'_{\cald,2}=\bk\prec\oplus\,\bk\succ,\]
and then extend $\gamma$ to $\calp'_\cald$ such that
\begin{align*}
&\gamma(\prec;I,\cali)=I,\,\gamma(\prec;\cali,I)=0,\\
&\gamma(\succ;I,\cali)=0,\,\gamma(\succ;\cali,I)=I.
\end{align*}
Let $\star_0:=\cali,\,\star_1:=I,\,\star_2:=\prec+\succ$ and $\star_n:=\gamma(\star_2;\star_{n-1},I)$ inductively for any $n\geq3$. Then $\star_2$ clearly satisfies Eq.~\meqref{eq:star2}.
By the relations of $\calp'_\cald$, we know that $\star_2$ is also associative.
According to Lemma~\mref{le:ass}, the sequence of operations $\{\star_n\}_{n\geq0}$ satisfies Eqs.~\meqref{eq:starn} and
\meqref{eq:ass}.
\end{exam}

Now for the braided dendriform operad case, we further have the following property.
\begin{prop}\mlabel{prop:den-coherent}
The stated unit action on $\calp_\cald$ generated by $M_2=\{\prec,\succ\}$ is coherent.
\end{prop}
\begin{proof}
It is enough to check that the operations $\prec$ and $\succ$ on $A\boxtimes B$ defined by Eq.~\meqref{eq:tpp-alg} satisfy the dendriform relations and the compatibility conditions \meqref{bda1}, \meqref{bda2} for any $\calp_\cald$-algebras $A,B$. The proof is the same as the case of the algebraic operads proved in~\cite[Proposition~4.9]{GL0} to which we refer the reader for details.
\end{proof}

Let $\calp$ be a (braided) nonsymmetric operad equipped with a coherent unit action and
$A^+:=\bk\oplus A$ be a unitary $\calp$-algebra. Do the same for $A^+\otimes A^+\cong(A\boxtimes A)^+$ via Eq.~\meqref{eq:tpp-alg}.
Let $\Delta:A^+\rar (A\boxtimes A)^+\cong A^+\otimes A^+$ be a coassociative linear map, such that
$\Delta(1)=1\otimes 1$ and $\Delta'(a):=\Delta_A(a)-a\otimes 1-1\otimes a\in A\otimes A$ for any $a\in A$.

If $\Delta$ is a morphism of unitary $\calp$-algebras, we call $A^+$ together with $\Delta$ a {\bf $\calp$-bialgebra}.
Moreover, if there is a direct sum $A^+=\bigoplus_{n\geq0}A^{(n)}$ of vector spaces, such that $A^{(0)}=\bk$ and
\[\Delta(A^{(n)})\subseteq \sum_{i=0}^n A^{(i)}\otimes A^{(n-i)}\]
for all $n\geq0$, we call $A^+$ a {\bf connected graded $\calp$-Hopf algebra}.

As a braided analogue of \cite[Lemma~6]{Hol2}, we provide the following theorem.

\begin{theorem}\mlabel{thm:p-alg}
Let $\calp$ be a braided nonsymmetric operad equipped with a coherent unit action and let
$A^+:=\bk\oplus A$ with $A:=\calp(V)$ be the free unitary $\calp$-algebra generated by a braided vector space $V$. Then there is a coassociative morphism $\Delta_A:A^+\rar A^+\otimes A^+$ of $\calp$-algebras defined by
\[\Delta_A(x)=x\otimes 1+1\otimes x,\,x\in V,\]
which provides $A^+$ with a connected graded $\calp$-bialgebra algebra structure. It further provided $A^+$ with a connected braided bialgebra structure and hence a Hopf algebra structure.

Furthermore, if the Yang-Baxter operator $\sigma_A$ on $A^+$ becomes a  $\calp$-algebra isomorphism, then $\Delta_A$ is twisted cocommutative, i.e. $\sigma_A\Delta_A=\Delta_A$.
\end{theorem}

\begin{proof}
Since $\calp(V)$ is the free $\calp$-algebra on $V$ and $\calp(V)^{\boxtimes n}, n=2,3,$ are $\calp$-algebras by the coherent unit action, there is a unique $\calp$-algebra morphism
$$\Delta_A: \calp(V)\to \calp(V)^{\boxtimes 2}, v\mapsto 1\ot v +v\ot 1, v\in V, $$
identifying $V$ with $\calp_1\otimes V$. It then induces an braided (associative) algebra homomorphism
$\Delta: (\calp(V),\star) \to (\calp(V)^{\boxtimes 2}, \star).$
By the same argument, we obtain the compatibility requirement
\[\Delta_A\theta_A=\theta_{A\boxtimes A}\calp(\Delta_A),\]
where $\theta_A=\gamma_V$ and $\theta_{A\boxtimes A}$ are defined in Eq.~\meqref{eq:tpp-alg}.
Then it also guarantees the commutativity between $\Delta_A$ and the braidings,
\[(\Delta_A\otimes\Delta_A)\sigma_A=\beta_{22}(\Delta_A\otimes\Delta_A).\]

On the other hand, since
\[(\Delta_A\otimes\id_{A^+})\Delta_A(x)=x\otimes 1\otimes 1+1\otimes x\otimes 1+1\otimes 1\otimes x=(\id_{A^+}\otimes\Delta_A)\Delta_A(x)\]
for any $x\in V$, by the universal property of $\calp(V)$ again, $\Delta_A$ as a morphism of $\calp$-algebras is coassociative.
Also, $\Delta_A(1)=1\otimes 1$ and $\Delta'_A(a):=\Delta_A(a)-a\otimes 1-1\otimes a\in A\otimes A$ for any $a\in A$.
Furthermore, the free unitary $\calp$-algebra $A^+$ allows a grading $A^{(n)}:=\calp_n\otimes V^{\otimes n}$ which is respected by all $\calp$-operations. The same holds for $A^+\otimes A^+$. Thus $\Delta_A$ makes $A^+$ into a connected graded $\calp$-bialgebra algebra. Thus $(A^+,\star,\Delta)$ is a connected graded braided bialgebra and hence a braided Hopf algebra.

Finally, the stated cocommutativity of $\Delta_a$ follows from its definition.
\end{proof}

\begin{exam}\mlabel{ex:den-p-alg}
For the braided dendriform operad, by Proposition~\mref{prop:den-coherent} and Theorem~\mref{thm:p-alg}, we know that
the free unitary $\calp_\cald$-algebra over a braided vector space is a connected graded $\calp_\cald$-Hopf algebra. In particular, we recover the braided Hopf algebra structures of free braided dendriform algebras, namely, the braided analogue of the Loday-Ronco Hopf algebra of planar binary rooted trees, directly verified in \cite{GL0}.
\end{exam}

\begin{prop}
The conclusion of Theorem \mref{thm:p-alg} holds when the associativity condition stated in Lemma~\mref{le:ass} is replaced by a {\bf completely commutativity condition} for the braided  operad $(\calp,\gamma,I)$, in the sense that there are operations $\star_n\in\calp(n),\,n\geq0$, invariant under the right $\BB_n$-module action and satisfying Eqs.~\meqref{eq:starn} and \meqref{eq:star2}.
\mlabel{pp:p-comalg}
\end{prop}

\begin{proof}
Under this condition, $\calp$-operations on $A\boxtimes B$ are still well-defined via Eq.~\meqref{eq:tpp-alg}, thus again endow $A\boxtimes B$ with a $\calp$-algebra structure given by the same $\theta_{A\boxtimes B}$ from $\calp(A\boxtimes B)=\bigoplus_{n\geq0}\calp(n)\otimes_{\BB_n}(A\boxtimes B)^{\otimes n}$ to $A\boxtimes B$ in Eq.~(\mref{eq:tpp-alg}). In short, there is again a coherent unit action on $\calp$. Then Proposition~\mref{pp:p-comalg} follows from the same argument.
\end{proof}

\begin{exam}
Typical examples of completely commutative (braided) operads with coherent unit actions are the (braided) commutative dendriform (identified with Zinbiel) operad and the braided commutative tridendriform (CTD) operad.
Thus Proposition~\mref{pp:p-comalg} in particular implies the braided Hopf algebra structures on free braided commutative (tri)dendriform algebras. Free braided commutative dendriform algebras are shown~\mcite{GL} to be isomorphism to the quantum shuffle algebras defined in \mcite{Ro} for involutive braidings. Thus the coherent unit action approach of braided Hopf algebras in Proposition~\mref{pp:p-comalg} puts the braided Hopf algebra structure on quantum shuffle algebras in a broader context. A similar approach can be given to the quantum quasi-shuffle algebra in~\mcite{JRZ}.
\end{exam}

\section{Classification of braided binary quadratic nonsymmetric operads with coherent unit actions}\mlabel{sec:clas}
In \mcite{EG} the authors classified all the associative, binary, quadratic and nonsymmetric (ABQR) operads with coherent unit actions. In this section we extend this classification to the braided context.

\begin{theorem}
Let $(\calp,\gamma,I,\star)$ be a braided ABQR operad.
A unit action on $\calp$ is coherent
if and only if, for every $(\sum_i\dfoa_i\otimes \dfob_i,\sum_j\dfoc_j\otimes \dfod_j)\in\calp_2^{\otimes 2}\oplus\calp_2^{\otimes 2}$ from the quadratic relations of $\calp$,
the following {\bf coherence equations} hold in terms of the linear maps $\alpha,\beta$ defined in Eq.~\meqref{eq:ua}.
\begin{enumerate}
\item[\bf{(C1)}] $\sum_i \beta(\dfoa_i)\dfob_i = \sum_j \beta(\dfoc_j)\dfod_j,$
\item[\bf{(C2)}] $\sum_i \alpha(\dfoa_i)\dfob_i  =\sum_j \beta(\dfod_j) \dfoc_j, $
\item[\bf{(C3)}] $\sum_i \alpha(\dfob_i)\dfoa_i =\sum_j \alpha(\dfod_j) \dfoc_j,$
\item[\bf{(C4)}] $\sum_i \beta(\dfob_i)\dfoa_i=\sum_j \beta(\dfoc_j)\beta(\dfod_j) \star, $
\item[\bf{(C5)}] $\sum_i \alpha(\dfoa_i)\alpha(\dfob_i) \star =
    \sum_j \alpha(\dfoc_j) \dfod_j.$
\end{enumerate}
\label{thm:coherent}
\end{theorem}
\begin{proof}
A braided ABQR operad $\calp$ is generated by
$\calp_1=\bk I,\,\calp_2=\bk\omega$
with relations
\[\sum_i\gamma(\dfoa_i;\dfob_i,I)=\sum_j\gamma(\dfoc_j;I,\dfod_j),\]
where $(\sum_i\dfoa_i\otimes \dfob_i,\sum_j\dfoc_j\otimes \dfod_j)\in\calp_2^{\otimes 2}\oplus\calp_2^{\otimes 2}$.

A unit action on $\calp$ is coherent
if and only if, for any $\calp$-algebra $A$ and $B$ and for any $a,a',a''\in A^+$ and $b,b',b''\in B^+$, we have the equation
\begin{equation}
\sum_i \left((a\ot b) \dfoa_i (a'\ot b')\right) \dfob_i (a''\ot b'')
    = \sum_j (a\ot b) \dfoc_j \left( (a'\ot b')\dfod_j (a''\ot b'')\right).
\label{eq:coh}
\end{equation}
Then there are eight mutually disjoint cases for eight subsets of
$b,b',b''$ which are in $\bfk$ (and hence not in $B$). Note that when all of $b,b',b''$ are in $\bk$ or all of $b,b',b''$ are in $B$, Eq. (\ref{eq:coh}) just means that $(\sum_i \dfoa_i\otimes
\dfob_i,\sum_j \dfoc_j\ot \dfod_j)$ is a relation for $\calp$, so is
automatic true. Thus to prove the theorem we only need to prove

\begin{enumerate}
\item[] Case 1.  Eq. (\ref{eq:coh}) holds for $b\in\bk$, $b',b''\in B$ if and only
if {\bf(C1)} is true;
\item[] Case 2. Eq. (\ref{eq:coh}) holds for $b'\in\bk$, $b,b''\in B$ if and only
if {\bf(C2)} is true;
\item[] Case 3. Eq. (\ref{eq:coh}) holds for $b''\in\bk$, $b,b'\in B$ if and only
if {\bf(C3)} is true;
\item[] Case 4. Eq. (\ref{eq:coh}) holds for $b,b'\in\bk$, $b''\in B$ if and only
if {\bf(C4)} is true;
\item[] Case 5. Eq. (\ref{eq:coh}) holds for $b',b''\in\bk$, $b\in B$ if and only
if {\bf(C5)} is true;
\item[] Case 6. Eq. (\ref{eq:coh}) holds for $b,b''\in\bk$, $b'\in B$ if
{\bf(C1)} is true.
\end{enumerate}
All these cases are routinely checked as in the original paper \mcite{EG} using binary operations on $A\boxtimes B$ defined in Eq.~\meqref{eq:tpp-alg}.
\end{proof}

We now apply Theorem~\ref{thm:coherent} to classify all braided ABQR operads with coherent
unit actions.
\begin{theorem} Let $(\calp,\gamma,I,\star)$ be a braided ABQR operad of dimension $n$
(that is, $\dim\calp_2=n$).
\begin{enumerate}
\item There is a coherent unit action $(\alpha,\beta)$ on $\calp$ with
$\alpha\neq \beta$ if and only if there is
a basis $\{ \dfop_i\}$ of $\calp_2$ with $\star=\sum_i \dfop_i$ such that
the space $\Lambda$ of quadratic relations for $\calp$
are contained in the subspace $\Lambda_{n,\,\coh}'$ of
$\calp_2^{\ot 2} \oplus \calp_2^{\ot 2}$ with the basis
\begin{equation}
\lambda'_{n,\,\coh} := \left\{ \begin{array}{l}
(\star \ot \dfop_2, \dfop_2\ot \dfop_2), \\
 (\dfop_1 \ot \dfop_1, \dfop_1 \ot \star), \\
 (\dfop_i\ot \dfop_1,  \dfop_i\ot \dfop_1),\ 2\leq i\leq n, \\
 (\dfop_2 \ot \dfop_j, \dfop_2\ot \dfop_j),\ 3\leq j\leq n, \\
 (\dfop_1\ot \dfop_i, \dfop_i \ot \dfop_2),\ 3\leq i\leq n,\\
 (\dfop_i\ot \dfop_j, 0), \
 (0, \dfop_i \ot \dfop_j),\ 3\leq i,j\leq n.
\end{array} \right\}
\mlabel{eq:an=b}
\end{equation}
\mlabel{it:cohneq}
\item There is a coherent unit action $(\alpha,\beta)$ on $\calp$ with
$\alpha= \beta$ if and only if there is a basis $\{ \dfop_i\}$ of $\calp_2$
with $\star=\sum_i \dfop_i$ such that the space $\Lambda$ of quadratic relations for $\calp$
are contained in the subspace $\Lambda_{n,\,\coh}''$ of
$\calp_2^{\ot 2} \oplus \calp_2^{\ot 2}$ with the basis
\begin{equation}
\lambda_{n,\,\coh}'':= \left \{ \begin{array}{l}
(\dfop_1 \ot \star, \dfop_1\ot\star) +(\star\ot \dfop_1, \star \ot \dfop_1)
 - (\dfop_1 \ot \dfop_1, \dfop_1 \ot \dfop_1), \\
 (\dfop_i\ot \dfop_j, 0), \
 (0, \dfop_i \ot \dfop_j),\ 2\leq i,j\leq n.
\end{array} \right \}.
\mlabel{eq:a=b}
\end{equation}
\mlabel{it:coheq}
\end{enumerate}
\mlabel{thm:crel}
\end{theorem}
\begin{proof}
The proof is the same as the one for the algebraic operads in~\cite[Theorem~4.9]{EG}, so we refer the reader there for details.
\end{proof}

By inspection, the braided dendriform algebra (see Proposition~\mref{prop:den-coherent}) and braided tridendriform algebras have coherent unit actions.

\smallskip

\noindent
{\bf Acknowledgments.}
We thank the organizers for the hospitality during the International Workshop on Hopf Algebras and Tensor Categories held in Nanjing University, September 2019.
This work is supported by Natural Science Foundation of China (Grant Nos. 11501214, 11771142,  11771190).

\bibliographystyle{amsplain}

\end{document}